\documentclass[12pt, reqno]{amsproc}
\usepackage{graphicx, enumerate}
\usepackage[margin=1in]{geometry}
\usepackage{amsmath}
\usepackage{amsfonts}
\usepackage{amssymb}
\usepackage{amsfonts,amsmath,amssymb}
\usepackage{amsthm}
\usepackage[symbol]{footmisc}

\newtheorem{defi}{Definition}
\newtheorem{lemma}{Lemma}
\newtheorem{theorem}{Theorem}

\newtheorem{crlly}{Corollary}
\newtheorem{eg}{Example}
\author[Naveen Mehra, Garima Pant and S. K. Chanyal]{Naveen Mehra, Garima Pant and S. K. Chanyal}
\address{Naveen Mehra, Department of Mathematics, Kumaun University, D.S.B. Campus, Nainital-263001, Uttarakhand, India.}
\email{naveenmehra00@gmail.com}
\address{Garima Pant, Department of Mathematics, University of Delhi, Delhi-110007, India.}
\email{garimapant.m@gmail.com}
\address{S. K. Chanyal, Department of Mathematics, Kumaun Univesity, D.S.B. Campus, Nainital-263001, Uttarakhand, India.}
\email{skchanyal.math@gmail.com}
\thanks{The research work of the first author is supported by a research fellowship from the Council of Scientific and Industrial Research(CSIR), New Delhi, India.}
\thanks{The research work of the second author is supported by a research fellowship from the University Grants Commission (UGC), New Delhi, India. }
\title[]{Growth of Solutions of Complex Differential Equations with Entire Coefficients having a Multiply-Connected Fatou Component} 
 \begin{document}
\maketitle
\begin{abstract}
In this study, we show that all non-trivial solutions of $f''+A(z)f'+B(z)f=0$
have infinite order, provided that the entire coefficient $A(z)$ has certain restrictions and $B(z)$ has multiply-connected Fatou component. We also extend these results to higher order linear differential equations. 
\end{abstract}
\section{Introduction}
In this paper, we discuss the growth of solutions of 
\begin{equation}\label{eqcde}
f''+A(z)f'+B(z)f=0,
\end{equation}
where $A(z)$ and $B(z)$ are entire functions. We use the notations $\lambda(f)$, $\rho(f)$ and $\mu(f)$ for  the exponent of convergence, order of growth and lower order of growth of the entire function $f(z)$ respectively. Since the coefficients are entire functions, all the solutions are also entire functions. If at least one of the coefficients is a transcendental entire, then almost all the non-trivial solutions are of infinite order. The main focus of this study is to find restrictions on $A(z)$ and $B(z)$ so that all non-trivial solutions of equation \eqref{eqcde} have infinite order. We collate some well-known results in the following theorem.\\

\textbf{Theorem 1.}\label{gg} \textit{Suppose that $A(z)$ and $B(z)$ are non-constant entire functions. Then, all non-trivial solutions of equation \eqref{eqcde} are of infinite order, if one of the following holds:
\begin{enumerate}[(i)]
\item\cite{gundersen2} $\rho(A) <\rho(B)$;
\item\cite{gundersen2} $A(z)$ is a polynomial and $B(z)$ is transcendental;
\item\cite{heller} $\rho(B)<\rho(A)\leq\frac{1}{2}$.
\end{enumerate}}

\begin{eg}
The differential equation $f''+(ze^{-z}-z-1)f'+ zf=0$, $\rho(ze^{-z}-z-1)>\rho(z)$ has a solution $f(z)=e^z-1$ of order $1$. 
\end{eg}
\begin{eg}
The order of growth of coefficients of the differential equation $f''-ze^{-z}f'+e^{-z}f=0$ is equal to $1$ and it has a solution $f(z)=z$ of order $0$.
\end{eg} 
In the above two examples we observe that $\rho(B)\leq\rho(A)$ and the differential equation \eqref{eqcde} has a non-trivial solution of finite order. Here, we are trying to find the conditions on coefficients $A(z)$ and $B(z)$ such that all non-trivial solutions of equation \eqref{eqcde} are of infinite order. Before stating our main results, we invoke the definition of Fatou set for the meromorphic function $f$.
\begin{defi}
The set of family of iterates $\{f^n:n\in\mathbb{N}\}$ for the transcendental meromorphic function $f$  is called a Fatou set if it is well-defined and forms a normal family.  
\end{defi}

In this paper, we prove the following results.
\begin{theorem}\label{mainth1}
Let $A(z)$ satisfy $\lambda(A)<\rho(A)$ and $B(z)$ be a transcendental entire function with a multiply-connected Fatou component. Then, all non-trivial solutions of equation (\ref{eqcde}) have infinite order of growth.\end{theorem}

We illustrate our first result with an example.

\begin{eg}
Consider the function
\begin{equation}\label{egmcfc}
B(z)=Cz^2 \displaystyle{ \prod_{n=1}^{\infty}\left(1+\frac{z}{a_n}\right)},
\end{equation}
where $C = (4e)^{-1}$, and $a_n$ satisfies $1 < a_1 < a_2 < ...$ and grows rapidly such that $a_{n+1} < B(a_n) <2a_{n+1}$. This function was constructed by Baker \cite{baker}. Here, $B(z)$ has a multiply-connected Fatou component. Let $A(z)=\sin ze^{z^2}$. Then using Theorem \ref{mainth1}, all non-trivial solutions of equation \eqref{eqcde} are of infinite order.
\end{eg}

\begin{theorem}\label{mainth2}
Let $A(z)$ be an entire function having a finite deficient value and  $B(z)$ be a transcendental entire function with a multiply-connected Fatou component. Then, all non-trivial solutions $f$ of equation \eqref{eqcde} are of infinite order.
\end{theorem}

Let $A(z)$ be an entire function of finite genus $p\geq 1$ and all of its zeros lie in the sector $$\left\{z:|\arg z|\leq \frac{\pi}{2(p+1)}\right\}.$$ T. Kobayashi \cite{kobayashi} proved that such a function has zero as a deficient value. Note that the genus of an entire function $f(z)$ having the representation
$$f(z) = P(z)e^{Q(z)},$$ where $$P(z)=z^m\prod\limits_{n=1}^\infty\left(1-\frac{z}{z_n}\right)\exp\left(\frac{z}{z_n}+\cdots+\frac{1}{p}\left(\frac{z}{z_n}\right)^p\right)$$ such that $$\sum_{n=1}^{\infty}\frac{1}{|z_n|^{p+1}}$$ converges. Here, $p$ is an integral and $z_n$'s are roots of $P(z)$ and root of multiplicity $m$ at point $0$ and $deg(Q(z))=q$. Then, the genus is $\max(p,q)$.\\

We state a consequence of Theorem \ref{mainth2}, which is as follows:
\begin{crlly}
Let $A(z)$ be an entire function of finite genus $p\geq 1$, and let all of its zeros lie in the sector 
$$\left\{z:|\arg z|\leq \frac{\pi}{2(p+1)}\right\}.$$
Suppose that $B(z)$ is a transcendental entire function with a multiply-connected Fatou component. Then, every non-trivial solution of equation \eqref{eqcde} has infinite order of growth.  
\end{crlly}
\begin{eg}
We illustrate Theorem B by an example. The differential equation $$f''-e^zf'+B(z)f=0,$$ where $-e^z$ has  zero as a finite deficient value and $B(z)$ is a transcendental entire function satisfying equation \eqref{egmcfc}.
\end{eg}


Before stating our final result, we would like to recall the following definition.
\begin{defi}
For an entire function $A(z)$ of order $\rho$ with
respect to the proximate order $\rho(r)$, the indicator $h(\theta)$ is defined by
$$h(\theta) = \lim_{r\to\infty}\sup\frac{\log|A(re^{\iota\theta})|}{r^{\rho(r)}},$$
where $\rho(r)\to\rho$ as $r\to\infty$. The function $A(z)$ is said to have a
completely regular growth if there exist disks $D(a_k, s_k)$ satisfying
$$\sum_{|a_k|\leq r}s_k = o(r),$$
such that
$$\log|A(re^{(\iota\theta)})|=h(\theta)r^{\rho(r)}+o(r^{\rho(r)}),$$
$$re^{\iota\theta}\not\in\cup_kD(a_k,s_k)$$
as $r\to\infty$, uniformly in $\theta.$ The set formed by the union of disks $\cup_kD(a_k,s_k)$ is called $C_0$-set.
\end{defi}
More details can be obtained from the book by B. Ja. Levin\cite{levin}. Our final result is as follows:\\

\begin{theorem}\label{mainthm4}
Let $A(z)$ be a completely regular growth entire function. The set
$F = \{\theta\in [0, 2\pi) : h(\theta) = 0\}$ is of zero Lebesgue measure. Let $B(z)$ be a transcendental entire function with a multiply-connected Fatou component and $\rho(A)\neq \rho(B)$. Then,
every non-trivial solution of equation \eqref{eqcde} is of infinite order.
\end{theorem}
Consider an exponential polynomial of order $n$: 
\begin{equation}\label{expopoly}
A(z)= P_1(z)e^{Q_1(z)} +\cdots + P_k(z)e^{Q_k(z)},
\end{equation}
where $P_i$ and $Q_i$ are polynomials with $i=0,1,\ldots, n$. Heittokangas et al.\cite{heitto} (see Lemma 1.3) proved that an exponential polynomial of the form \eqref{expopoly} has a completely regular growth.
\begin{crlly}
Let $A(z)$ be an exponential polynomial of the form \eqref{expopoly} of order $n$ and $B(z)$ be a transcendental entire function with a multiply-connected Fatou component. Then, every solution $f(\not\equiv 0)$ of equation \eqref{eqcde} satisfies $\rho(f)=\infty$.
\end{crlly}  

In our first result, namely Theorem \ref{mainth1}, we consider $A(z)$ to be an entire function with $\lambda(A)<\rho(A)$. This implies that we can write $A(z)=h(z)e^{P(z)}$, where $h(z)$ is a transcendental entire function such that $\rho(h)<\rho(A)$ and $P(z)=a_nz^n+a_{n-1}z^{n-1}+\cdots + a_0$ is a non-constant polynomial of degree $n$. Some results have already been proved by considering $A(z)$ to be an entire function satisfying $\lambda(A)<\rho(A)$ and $B(z)$ to be an entire function with certain restrictions. We summarize some of the these results in the following theorem.\\

\textbf{Theorem 2.} \textit{Let $A(z)=h(z)e^{P(z)}$ satisfy $\lambda(A)< \rho (A)$ and $B(z)$ satisfy any of the following conditions:
\begin{enumerate}[(i)]
\item\cite{longshiwuzhang} $B(z)=b_mz^m+b_{m-1}z^{m-1}+\cdots + b_0$ is a polynomial of degree $m$ such that
\begin{enumerate}[(a)]
\item $m + 2 < 2n$, or
\item $m + 2 > 2n$ and $m + 2 \neq 2kn$ for all integers $k$, or
\item $m + 2 = 2n$ and $\frac{a_n^2}{b_m}$ is not real and negative.
\end{enumerate}
\item\cite{kumarsaini} $\rho(B) \neq \rho(A)$;
\item\cite{kumarsaini} has Fabry gap. 
\end{enumerate} 
Then, all non-trivial solution of \eqref{eqcde} have infinite order.}\\

Some researchers have studied the growth of non-trivial solutions of equation \eqref{eqcde}, when $A(z)$ is an entire function having a finite deficient value and $B(z)$ satisfies any of the conditions given in the following theorem. Theorem \ref{mainth2} is inspired by this result.\\

\textbf{Theorem 3.} \textit{Let $A(z)$ be an entire function having a finite deficient value and $B(z)$ satisfy any of the following conditions:
\begin{enumerate}[(i)]
\item\cite{kwonkim} $\rho(B)\leq\frac{1}{2}$ and $\rho(A)\neq \rho(B)$;
\item\cite{wuzhu} $\mu(B)< \frac{1}{2}$;
\item\cite{wuzhu} there exist two constants $\alpha>0$, $\beta>0$, for any given $\epsilon>0$, two finite sets of real numbers $\{\phi_k\}$ and $\{\theta_k\}$ that satisfy $\phi_1 < \theta_1 < \phi_2 < \theta_2 < \cdots < \phi_m < \theta_m < \phi_{m+1} (\phi_{m+1}=\phi_1+2\pi)$ and
$$\sum_{k=1}^m(\phi_{k+1}-\theta_k)<\epsilon,$$
such that
$$|B(z)|\geq \exp\{(1+o(1))\alpha |z|^\beta\}$$
as $z\to\infty$ in $\phi_k \leq arg z \leq \theta_k (k = 1, 2,\cdots ,m)$;
\item\cite{wuzhu} for any $k>0$ 
$$\lim_{|z|=r\to\infty}\frac{|B(z)|}{r^k}=\infty$$ holds outside a set $G$ of $r$-values with $m_l(G)<\infty$.
\end{enumerate}
Then, all non-trivial solutions of equation \eqref{eqcde} are of infinite order.}\\

Our last result is motivated by the following result, which is a combination of the results proved by G. Zhang \cite{zhang}. He considered $A(z)$ to be an entire function of completely regular growth and the indicator $h(\theta)=0$ in a set of zero Lebesgue measure and $B(z)$ with some restrictions as given in the next result.\\
 
\textbf{Theorem 4.}\cite{zhang}\textit{ Let $A(z)$ be an  entire
function with completely regular growth and the set $F = \{\theta\in [0, 2\pi) : h(\theta) = 0\}$
be of zero Lebesgue measure. Let B(z) be a transcendental
entire function satisfying $\rho(A)\neq\rho(B)$ and satisfy any of the following conditions:
\begin{enumerate}[(i)]
\item $\mu(B)<\frac{1}{2}$;
\item has Fabry gaps;
\item $T(r,B)\sim\log \alpha M(r,B)$ as $r\to\infty$, where $0<\alpha\leq 1$, outside a set of finite logarithmic measure.
\end{enumerate}
Then, all non-trivial solutions of equation \eqref{eqcde} are of infinite order.}\\

In the previous result, notation $T(r,B)$ is the characteristic function of the entire function $B(z)$ and $M(r,B)=\max_{|z|=r}|B(z)|$. $T(r,B)\sim\log \alpha M(r,B)$  is $$\lim_{r\to\infty}\frac{T(r,B)}{\alpha M(r,B)}=1$$  for $0<\alpha\leq 1$ and $r\to\infty$. Recall that, if an entire function $\displaystyle{f(z)=\sum\limits_{n=0}^{\infty} a_{\lambda_n}z^{\lambda_n}}$ if $\displaystyle{\frac{\lambda_n}{n}}$ diverges to $\infty$, it has a Fabry gap.\\

If the hypothesis of the theorems are relaxed, the results will not hold. This is illustrated through the following examples. 
\begin{eg}
\begin{enumerate}[(a)]
\item The differential equation $f''(z)-e^zf'(z)+ e^zf=0$ has a finite order solution $f(z)=e^z-1$. Here, $A(z)=-e^z$ satisfies the conditions of Theorems $\rm{\ref{mainth1}}$ 
and $\rm{\ref{mainthm4}}$. However, $B(z)=e^z$ has no multiply-connected Fatou component.


\end{enumerate}
\end{eg}


\section{Auxiliary Results} 
To make this article self-contained, we include some important results that are used in the proof of our theorems. Before stating these lemmas, we recall the definition of some terms that are used in our results. The linear measure of a set $G\subset[0,\infty)$ is 
$$m(G)=\int_Gdt. $$
The logarithmic measure of a set $H\subset [1,\infty)$ is 
$$m_l(F)=\int_E\frac{dt}{t}.$$
The lower logarithmic density and upper logarithmic density of a set $H\subset [1,\infty)$ are 
$$\underline{\log dens }(H) = {\lim\inf}_{r\to\infty}\frac{m(H\cap [1,r]}{\log r}$$ and
$$\overline{\log dens }(H) = {\lim\sup}_{r\to\infty}\frac{m(H\cap [1,r]}{\log r}$$
respectively. $M(r,f)$ and $L(r,f)$ are the maximum modulus and minimum modulus of $f(z)$ respectively on $|z|=r$.\\
 
The following lemma is used extensively to prove results  in complex differential equations. It gives an estimate of the logarithmic derivative of the finite order transcendental meromorphic function.
\begin{lemma}\cite{gundersen}\label{lgundersen} Let $f$ be a transcendental meromorphic function with finite order and $(k, j)$ be a finite pair of integers that satisfies $k > j \geq 0.$ Let $\epsilon>0$ be a given constant. Then, the following three statements hold:
\begin{enumerate}
\item[(i)] there exists a set $E_1 \subset [0, 2\pi)$ that has linear measure zero, such that: if $\psi_0 \in [0, 2\pi) -E_1$, there is a constant $R_0 = R_0(\psi_0) > 0$ such that for all z satisfying $\arg z = \psi_0$ and $|z| \geq R_0$ and for all $(k, j) \in \Gamma$,
\begin{equation}\label{f"byf}
\left|\frac{f^{(k)}(z)}{f^{(j)}(z)}\right| \leq |z|^{(k-j)(\rho-1+\epsilon)}.
\end{equation}
\item[(ii)] there exists a set $E_2 \subset (1,\infty)$ that has finite logarithmic measure, such that for all z with $|z| \notin E_2 \cup[0, 1]$ and for all $(k, j) \in \Gamma$,
the inequality \eqref{f"byf} holds.
\item[(iii)] there exists a set $E_3 \subset [0,\infty)$ that has finite linear measure, such that for all z with $|z| \notin E_3$ and for all $(k, j) \in \Gamma$,
\begin{equation}\label{f"byf-2}
\left|\frac{f^{(k)}(z)}{f^{(j)}(z)}\right| \leq |z|^{(k-j)(\rho+\epsilon)}.
\end{equation}
\end{enumerate}
\end{lemma}

If a transcendental meromorphic function $f$ has a Baker wandering domain, then $\mathcal{J}$ has only bounded components\cite{zheng}. Baker\cite{baker2} showed that for a transcendental meromorphic function, every multiply-connected Fatou component has a Baker wandering domain. Thus, every multiply-connected Fatou component of a transcendental meromorphic function $f$ has only bounded components. So, the following lemma, proved by Zheng\cite{zheng}, can be applied for a  transcendental meromorphic function having a multiply-connected Fatou component.

\begin{lemma}\cite{zheng}\label{lzheng}
Suppose $f$ is a transcendental meromorphic function having at most finite poles. If $\mathcal{J}(f)$ has only bounded components, then for any complex number, there exists a constant $0<\beta<1$ and two sequences of positive numbers $\{r_n\}$ and $\{R_n\}$ with $r_n\to\infty$ and $R_n/r_n\to\infty(n\to\infty)$ such that
$$M(r,f)^{\beta}\leq L(r,f)\quad\mbox{for}\quad r\in H,$$
where $H=\cup_{n=1}^\infty\{r:r_n<r<R_n\}.$  
\end{lemma}
The following lemma is given by Bank, Laine and Langley\cite{banklainelangley}. It provides an estimate for an entire function with integral order except a set $E$ of zero linear measure, that is, $m(E)=\int_Edt=0$. 
\begin{lemma}\cite{banklainelangley}\label{lcritical}
Assume $A(z) = h(z)e^{P(z)}$ is an entire function satisfying $\lambda(A) < \rho(A) = n$, where $P(z)$ is a polynomial of degree $n$. Then, for
every $\epsilon > 0$, there exists $E \subset [0, 2\pi)$ of linear measure zero satisfying
\begin{enumerate}[(i)]
\item for $\theta \in [0, 2\pi)\setminus E$ with $\delta(P, \theta) > 0$, there exists $R > 1$ such that
$$exp ((1 - \epsilon)\delta(P, \theta)r^n) \leq |A(re^{\iota\theta})|$$
for $r > R$,
\item for $\theta \in [0,2\pi)\setminus E$ with $\delta(P, \theta) < 0$, there exists $R > 1$ such that
$$|A(re^{\iota\theta})| \leq exp ((1 - \epsilon)\delta(P, \theta)r^n)$$ 
for $r > R$.
\end{enumerate}
\end{lemma}
The next lemma provides an estimate of the integral of the logarithmic derivative of the finite order meromorphic function.
\begin{lemma}\cite{fuchs}\label{lfuchs}
Let $f(z)$ be a meromorphic function with $\rho<\infty$. For given $\epsilon>0$ and $\alpha$, $0 <\alpha<1/2$, there exists $K(\rho,\xi)$ such that for all $r\in F\subset [0,\infty)$ of lower logarithmic density greater than $1-\xi$ and for every interval $J\subset [0, 2\pi)$ of length $\alpha$, 
$$r\int\limits_J\left|\frac{f'(re^{\iota\theta})}{f(re^{\iota\theta})}\right|d\theta < K(\rho,\xi)(\alpha \log\frac{1}{\alpha})T(r,f).$$
\end{lemma}
The following lemma gives the bounds of non-trivial solutions of equation \eqref{eqcde} and its derivatives in some sector, under certain conditions on $A(z)$ and $B(z)$.  
\begin{lemma}\cite{gundersen2}\label{lgundersen2}
Assume $A(z)$ and $B(z)(\not\equiv 0)$ to  be entire functions such that for $\alpha >0$, $\beta >0$, $\theta_1$, $\theta_2$ where $\theta_1 < \theta_2$,
$$|A(z)|\geq \exp \{(1+o(1)\alpha |z|^{\beta})\},$$
$$|B(z)|\leq \exp \{(1+o(1)|z|^{\beta})\}$$
as $z\to\infty$ and  $\theta\in [\theta_1, \theta_2]$. Let $S(\theta_1 + \epsilon, \theta_2 - \epsilon) = \{z: \theta_1+\epsilon \leq \arg z \leq \theta_2\ -\epsilon\}$ for a given small $\epsilon > 0$, assume that $f$ is a non-trivial solution of equation \eqref{eqcde} of finite order. Then, the following conclusions hold:
\begin{enumerate}[(i)]
\item
there exists a constant $b(\neq 0)$ such that $f(z)\to b$ as $|z|=r\to \infty$ in $S(\theta_1 + \epsilon, \theta_2 - \epsilon)$. Furthermore,
\begin{equation}
|f(z) - b| \leq \exp\{-(1 + o(1))\alpha |z|^\beta \}
\end{equation}
as $z \to\infty$ in $S(\theta_1 + \epsilon, \theta_2 - \epsilon)$.
\item
For each integer $k \geq 1$,
\begin{equation}
|f^{(k)}(z)| \leq \exp\{-(1 + o(1))\alpha |z|^\beta \}
\end{equation}
as $|z|=r \to\infty$ in $S(\theta_1 + \epsilon, \theta_2 - \epsilon)$.
\end{enumerate}
\end{lemma}
In the successive lemma, we get the modulus of angular sectors in which the zeros of a non-integral entire function lie.
\begin{lemma}\cite{kwon}\label{lkwon}
Assume $f(z)$ to be an entire function of non-integral order $\rho<\infty$ and of genus $p\geq1$. For any given $\epsilon>0$, let all the zeros of $f(z)$ have their arguments in the following subset of real numbers:
$$S(p,\epsilon)=\displaystyle{\{\theta:|\theta|\leq\frac{\pi}{2(p+1)}-\epsilon\}},$$
if $p$ is odd, and
$$S(p,\epsilon)=\displaystyle{\{\theta:\frac{\pi}{2p}+\epsilon\leq|\theta|\leq\frac{3\pi}{2(p+1)}-\epsilon\}},$$
if $p$ is even. Then, for any $\alpha >1$, there exists a real number $R>0$ such that $$|f(-r)|\leq\exp(-\alpha r^p)$$
for all $r\geq R$.
\end{lemma}

\section{Proof of theorems}

\subsection*{Proof of Theorem A}

Let $f$ be a finite order non-trivial solution of equation \eqref{eqcde}.
Using Lemma \ref{lgundersen}(ii), 
\begin{equation}\label{p11}
\left|\frac{f^{(k)}(re^{\iota \theta})}{f(re^{\iota \theta})}\right| \leq r^{k\rho(f)},
\end{equation}
for $r\not\in E_1\cup [0,1]$, where $E_1$ is a set with a finite logarithmic measure. Since $\lambda(A)<\rho(A)$, we have $A(z)=h(z)P(z)$, where $P(z)$ is a polynomial of degree $n. $Using Lemma \ref{lcritical}, for every $\epsilon>0$, there exists a set $E_2\subset[0,2\pi)$ of linear
measure zero such that for $\theta\in[0,2\pi)/E_2$ with $\delta(P,\theta)<0$ and $R_1>1$, 
\begin{equation}\label{p12}
|A(re^{\iota\theta})| \leq \exp ((1 - \epsilon)\delta(P, \theta)r^n)
\end{equation} 
 for $r > R_1$ . From Lemma \ref{lzheng},
\begin{equation}\label{fatoueq}
M(r,B)^\beta\leq |B(re^{\iota\theta})|
\end{equation}
for $0<\beta<1$ and $r\in H=\cup_{n=1}^\infty\{r:r_n<r<R_n\}$.

From equations \eqref{eqcde}, \eqref{p11}, \eqref{p12} and \eqref{fatoueq},
$$|B(z)|\leq\left|\frac{f''(re^{\iota\theta})}{f(re^{\iota\theta})}\right|+|A(z)|\left|\frac{f'(re^{\iota\theta})}{f(re^{\iota\theta})}\right|$$
$$M(r,B)^{\beta}<|B(z)|\leq (1+o(1))r^{2\rho(f)}$$
$$M(r,B)<(1+o(1))r^{4\rho(f)}$$
for large $r\in H\setminus E_1\cup[0,1]$ and $\theta\in[0,2\pi)\setminus E_2$.
This is a contradiction for a transcendental entire function.

\subsection*{Proof of Theorem B}
Suppose $f$ is a non-trivial solution of equation \eqref{eqcde} of finite order. Let $A(z)$ has a deficiency $\delta (a,f) = 2\delta >0$ at
$a\in \mathbb{C}$. Then, from the definition of deficiency,
$$m\left(r,\frac{1}{A-a}\right)\geq\delta T(r,A)$$
for all sufficiently large $r$. Hence, there exists a point $z_k$ such that $|z_r|=r$ and 
\begin{equation}\label{eqdefi1}
\log|A(z_r)-a|\leq - \delta T(r,A)
\end{equation}
for any sufficiently large $r$.\\
We first assume that $A(z)$ has a zero deficient value that is $a=0$. Now, fix $z_r=re^{\iota\theta_r}$ and let $\zeta >0$ be a sufficiently small number. Choose $\phi$ such that $\theta\in[\theta_r-\phi , \theta_r +\phi]$, where $|\theta_r-\phi|\leq\alpha$. Using Lemma \ref{lfuchs}, for sufficiently  small $\alpha$, $\xi = \epsilon/3$ and $r\in F$, where the lower logarithmic density of $F$ is greater than $1-\xi$,  
\begin{equation}\label{eqdefi2}
r\int\limits_J\left|\frac{A'(re^{\iota\theta})}{A(re^{\iota\theta})}\right|d\theta < cT(r,f),
\end{equation}
where $K(\rho,\xi)(\alpha\log\frac{1}{\alpha})<c$.\\
Combining equations \eqref{eqdefi1} and \eqref{eqdefi2}, for $\theta\in[\theta_r-\phi , \theta_r +\phi]$ and $r\in F$,
\begin{align*}
\log |A(re^{\iota\theta})|& = \log|A(re^{\iota\theta_r})|+\int\limits_{\theta_r}^{\theta}\frac{d}{dt}\log|A(re^{\iota t})|dt\\
& \leq -\delta T(r,A)+r\int\limits_{\theta_r}^{\theta}\left|\frac{A'(re^{\iota t})}{A(re^{\iota t})}\right||dt|\\
& \leq (- \delta + c)T(r,A) \leq 0.
 \end{align*}
If $A(z)$ has a finite deficient value $a\in\mathbb{C}$, we can consider zero as a deficient value for the function $A(z)-a$. Hence, there exist real numbers $\phi > 0$, $\theta_r$ and a set $F\subset [0,\infty )$ of lower logarithmic density greater than $1-\zeta$ such that for given $r\in F$ and $\theta\in [\theta_r - \phi , \theta_r + \phi ]$,
$$\log|A(re^{\iota\theta}) - a|\leq 0,$$
which implies that
\begin{equation}\label{eqdeficient}
|A(re^{\iota\theta})|\leq |a| + 1. 
\end{equation}
Using equation \eqref{eqcde},
\begin{equation*}
|B(z)|\leq\left|\frac{f''(re^{\iota\theta})}{f(re^{\iota\theta})}\right|+|A(z)|\left|\frac{f'(re^{\iota\theta})}{f(re^{\iota\theta})}\right|.
\end{equation*}
Using equations \eqref{p11}, \eqref{fatoueq} and \eqref{eqdeficient} in the above equation,
$$M(r,B)\leq r^{4\rho(f)}(2+ |a|),$$ 
for $r\in F \cup H\setminus ({E_1\cup[0,1]})$, and $\theta\in [\theta_r - \phi , \theta_r + \phi ]$. This is a contradiction for the transcendental entire function $B(z)$.




\subsection*{Proof of Theorem C}
Gundersen\cite{gundersen2} proved the desired result when $\rho(A)<\rho(B)$. So, we only need to prove the case when $\rho(A)>\rho(B)$. Suppose that $f$ is a non-trivial solution of equation \eqref{eqcde} such that $\rho(f)<\infty$.
Consider the set $$F^*=\{\theta\in[0,2\pi) :h(\theta)\leq 0 \}.$$ There are two cases depending on the Lebesgue measure of the set $F^*$.\\
\textbf{Case 1}: If $m(F^*)=0$, then the indicator of $A(z)$ satisfies $h(\theta)>0$ for every $\theta\in [0,2\pi)\setminus F^*$. From the definition of completely regular growth of $A(z)$,
\begin{equation}\label{eqcrg}
\log|A(re^{(\iota\theta)})|=h(\theta)r^{\rho(r)}+o(r^{\rho(r)}),
\end{equation}
for $\theta\in [0,2\pi)\setminus F^*$ and $z\not\in C_0$.\\
 Then, for any given $\delta \in (0,\frac{\pi}{4\rho(A)})$ and $\eta\in (0, \frac{\rho(A)-\rho(B)}{4})$,
\begin{equation}
|A(z)|\geq\exp\{(1 + o(1))\alpha|z|^{\rho(A)-\eta}\},
\end{equation}
\begin{align*}
|B(z)| & \leq \exp\{|z|^{\rho(B)+\eta}\}\\
       & \leq \exp\{|z|^{\rho(A)-2\eta}\\
       & \leq \exp\{o(1)|z|^{\rho(A)-\eta}
\end{align*}
as $z\not\in C_0$, $|z|=r\to\infty$ satisfying $\theta\in[0,2\pi)\setminus F^*$. Here, $\alpha$ is a positive constant depending on $\delta$. Then by Lemma \ref{lgundersen2}, for $z\not\in C_0$, $|z|=r \to\infty$ and $\theta\in[0,2\pi)\setminus F^*$,
\begin{equation}
|f(z) - b_j| \leq \exp\{-(1 + o(1))\alpha |z|^\beta \}.
\end{equation}
Therefore, $f$ is bounded in the whole complex plane by the Phragm$\acute{e}$n--Lindel$\ddot{o}$f principle. So, $f$ is a non-zero constant in the whole complex plane by Liouville's theorem. But $f$ cannot be a non-zero constant. This gives rise to a contradiction.\\

\textbf{Case 2.} If $m(F^*)>0$, then from equation \eqref{eqcde},
\begin{equation}\label{t4}
|B(z)|\leq\left|\frac{f''(re^{\iota\theta})}{f(re^{\iota\theta})}\right|+|A(z)|\left|\frac{f'(re^{\iota\theta})}{f(re^{\iota\theta})}\right|.
\end{equation}
Using equations \eqref{p11}, \eqref{fatoueq} and the fact that $h(\theta)<0$ for $\theta\in F^*$ in equation \eqref{eqcrg},
$$M(r,B)^\beta<(1+o(1))r^{2\rho(f)},$$
for large $r\in H\setminus(E_1\cup [0,1])$ and $z\not\in C_0$. This is a contradiction for a transcendental entire function.

 \section{Extension of Results to Higher Order}
Consider the $m^{th}$ order differential equation
\begin{equation}\label{cdehigherorder}
f^{(m)}+A_{m-1}(z)f^{(m-1)}+A_{m-2}(z)f^{(m-2)}+\ldots+A_1(z)f'+A_0(z)f=0,
\end{equation} 
where $A_{m-1}(z), \ldots, A_0(z)$ are entire functions. All the solutions of equation \eqref{cdehigherorder} are of finite order, if all the coefficients are polynomials. There are at most $j$ linearly independent solutions of finite order, if $A_j$ is the last transcendental entire function.

 Many researchers have worked on higher order linear differential equations, one may refer to \cite{chenshon, xuzhang}. In this section, we extend our second order differential equations results to higher order differential equations. In the higher order results, we consider $\rho(A_j)<\rho(A_0), j=0,1,2,\cdots,k-1,k+1,\cdots, m$. \\
In the next result we assume $A_k(z)$ to be an entire function satisfying $\lambda(A_k)<\rho(A_k)$ and  $A_0(z)$ to be an entire function having a multiply-connected Fatou component.
\begin{theorem}\label{mainthm1higherorder}
Suppose that $A_k(z)$ satisfies $\lambda(A_k)<\rho(A_k)$ and $A_0(z)$ is a transcendental entire function with a multiply-connected Fatou component satisfying $\rho(A_j)<\rho(A_0)$, where $j= 0,1,\ldots,k-1,k+1,\ldots, m$. Then, all non-trivial solutions of equation (\ref{cdehigherorder}) have infinite order of growth.
\end{theorem}
\begin{proof}
Let $f$ be a finite order non-trivial solution of equation \eqref{cdehigherorder}. We have $\lambda(A_k)<\rho(A_k)$. Thus, $A(z_k)=h(z)e^{P(z)}$, where $P(z)$ is a polynomial of degree $n$. Using Lemma \ref{lcritical}, for every $\epsilon>0$, there exists a set $E_2\subset[0,2\pi)$ of linear
measure zero such that for $\theta\in[0,2\pi)/E_2$ with $\delta(P,\theta)<0$ and $R_1>1$,
\begin{equation}\label{p12higherorder}
|A_k(re^{\iota\theta})| \leq \exp ((1 - \epsilon)\delta(P, \theta)r^n)
\end{equation} 
 for $r > R_1$ and $1\leq k\leq n$. From the definition of order of growth for $j=0,1,2,\ldots,k-1,k+1,\ldots, m-1$,
\begin{equation}\label{orderA_j}
|A_j(re^{r\iota\theta})|\leq \exp{r^{\rho(A_j) + \epsilon}} \leq \exp{r^{\rho(A_t) + \epsilon}},
\end{equation}
where $\rho(A_t)=\max\{\rho(A_j):1 \leq j\leq m; j\neq k\}$.
From Lemma \ref{lzheng},
\begin{equation}\label{fatoueqhigherorder}
M(r,A_0)^\beta\leq |A_0(re^{\iota\theta})|
\end{equation}
for $0<\beta<1$ and $r\in H=\cup_{n=1}^\infty\{r:r_n<r<R_n\}$.

From equations  \eqref{p11}, \eqref{cdehigherorder}, \eqref{p12higherorder},  \eqref{orderA_j} and \eqref{fatoueqhigherorder}, we have
$$|A_0(z)|\leq\left|\frac{f^{(m)}(re^{\iota\theta})}{f(re^{\iota\theta})}\right|+|A_{m-1}(z)|\left|\frac{f^{(m-1)}(re^{\iota\theta})}{f(re^{\iota\theta})}\right|+\cdots+|A_{k}(z)|\left|\frac{f^{(k)}(re^{\iota\theta})}{f(re^{\iota\theta})}\right|+\cdots+|A_0(z)|\left|\frac{f'(re^{\iota\theta})}{f(re^{\iota\theta})}\right|.$$
$$M(r,A_0)^{\beta}<|A_0(z)|\leq (1+(m-1)\exp{r^{\rho(A_t) + \epsilon}}+o(1))r^{m\rho(f)}$$
for large $r\in H\setminus E_1\cup[0,1]$ and $\theta\in[0,2\pi)\setminus E_2$.
This is not possible.
\end{proof}
\begin{theorem}\label{mainthm2higherorder}
Let $A_k(z)$ be an entire function having a finite deficient value and $A_0(z)$ be a transcendental entire function with a multiply-connected Fatou component satisfying $\rho(A_j)<\rho(A_0)$, where $j= 0,1,\ldots,k-1,k+1,\ldots, m$. Then, all non-trivial solutions of equation (\ref{cdehigherorder}) have infinite order of growth.
\end{theorem}
\begin{proof}
Suppose $f$ is a non-trivial solution of equation \eqref{cdehigherorder} of finite order. Let $A_k(z)$ have a deficiency $\delta (a,f) = 2\delta >0$ at
$a\in \mathbb{C}$. Then, from the definition of deficiency,
$$m\left(r,\frac{1}{A_k-a}\right)\geq\delta T(r,A_k)$$
for all sufficiently large $r$. Hence, there exists a point $z_k$ such that $|z_r|=r$. Using similar arguments as in the proof of Theorem \ref{mainth2} for $r\in F$,  where the lower logarithmic density of $F$ is greater than $1-\xi$ and choosing $\phi$ such that $\theta\in[\theta_r-\phi , \theta_r +\phi]$, where $|\theta_r-\phi|\leq\alpha$,
\begin{equation}\label{eqdeficienthigherorder}
|A_k(re^{\iota\theta})|\leq |a| + 1. 
\end{equation}
Using equations \eqref{p11}, \eqref{orderA_j}, \eqref{fatoueqhigherorder} and \eqref{eqdeficienthigherorder} in equation \eqref{cdehigherorder},
$$|A_0(z)|\leq\left|\frac{f^{(m)}(re^{\iota\theta})}{f(re^{\iota\theta})}\right|+|A_{m-1}(z)|\left|\frac{f^{(m-1)}(re^{\iota\theta})}{f(re^{\iota\theta})}\right|+\cdots+|A_{k}(z)|\left|\frac{f^{(k)}(re^{\iota\theta})}{f(re^{\iota\theta})}\right|+\cdots+|A_0(z)|\left|\frac{f'(re^{\iota\theta})}{f(re^{\iota\theta})}\right|$$
$$M(r,A_0)^\beta \leq r^{m\rho(f)}(2+ |a|+(m-1)\exp{r^{\rho(A_t) + \epsilon}}),$$ 
for $r\in F \cup H\setminus {E_1\cup[0,1]}$, and $\theta\in [\theta_r - \phi , \theta_r + \phi ]$. This is a contradiction for the transcendental entire function $A_0(z)$.
\end{proof}

\begin{theorem}\label{mainthm4higherorder}
Let $A_k(z)$ be a completely regular growth entire function, the set
$F = \{\theta\in [0, 2\pi) : h(\theta) \geq 0\}$ be of zero Lebesgue measure and $A_0(z)$ be a transcendental entire function with a multiply-connected Fatou component satisfying $\rho(A_j)<\rho(A_0)$, where $j= 0,1,\ldots,k-1,k+1,\ldots, m$. Then, every non-trivial solution of equation \eqref{cdehigherorder} is of infinite order.
\end{theorem} 
\begin{proof}
Suppose that $f$ is a non-trivial solution of equation \eqref{cdehigherorder} such that $\rho(f)<\infty$. Since $m(F)=0$, we have $h(\theta)<0$ for $\theta\in[0,2\pi)\setminus F$. By definition of completely regular growth of $A(z)$,
\begin{equation}\label{eqcrghigherorder}
\log|A_k(re^{(\iota\theta)})|=h(\theta)r^{\rho(r)}+o(r^{\rho(r)})
\end{equation}
for $\theta\in [0,2\pi)\setminus F$ and $z\not\in C_0$. Then, from equation \eqref{cdehigherorder},
$$|A_0(z)|\leq\left|\frac{f^{(m)}(re^{\iota\theta})}{f(re^{\iota\theta})}\right|+|A_{m-1}(z)|\left|\frac{f^{(m-1)}(re^{\iota\theta})}{f(re^{\iota\theta})}\right|+\cdots+|A_{k}(z)|\left|\frac{f^{(k)}(re^{\iota\theta})}{f(re^{\iota\theta})}\right|+\cdots+|A_0(z)|\left|\frac{f'(re^{\iota\theta})}{f(re^{\iota\theta})}\right|.$$
Using equations \eqref{p11} and \eqref{fatoueqhigherorder} and the fact that $h(\theta)<0$ for $\theta\in[0,2\pi)\setminus F$ in equation \eqref{eqcrghigherorder},
$$M(r,A_0)^{\beta}<|A_0(z)|\leq (1+(m-1)\exp{r^{\rho(A_t) + \epsilon}}+o(1))r^{m\rho(f)}$$
for large $r\in H\setminus(E_1\cup [0,1])$ and $z\not\in C_0$. This is not possible.
\end{proof}
In his paper\cite{kwon}, Kwon considered $A(z)$ to be an entire function having non-integral order $\rho(A)>1$, with all its zeros fixed in a sector and $B(z)$ to be an entire function with  $0<\rho(B)<1/2$. Kumar et al.\cite{mehra} replaced the $0<\rho(B)<1/2$ condition in Kwon's result by Fabry gaps on $B(z)$.\\

\textbf{Theorem 5.}\label{thkwon}\textit{
Let $A(z)$ be an entire function of finite non-integral order with $\rho(A)>1$ such that all the zeros of $A(z)$ lie in the angular sector  $\theta_1<argz<\theta_2$ satisfying
$$\theta_2-\theta_1<\frac{\pi}{p+1},$$ if $p$ is odd, and
$$\theta_2-\theta_1<\frac{(2p-1)\pi}{2p(p+1)},$$ if $p$ is even. Here, $p$ is the genus of $A(z)$. Let $B(z)$ be an entire function satisfying the following condition: 
\begin{enumerate}[(i)]
\item\cite{kwon} $0<\rho(B)<\frac{1}{2}$;
\item\cite{mehra} has Fabry gaps.
\end{enumerate}
Then, all non-trivial solutions $f(z)$ of equation \eqref{eqcde} are of infinite order.}\\

Recently, Mehra and Pande\cite{mehrapande} replaced condition on $B(z)$ in Theorem 5 with multiply-connected Fatou component. Next result is it's extension to higher dimension differential equation.

\begin{theorem}\label{mainthm3higherorder}
Let $A_k(z)$ be an entire function of finite non-integral order with $\rho(A_k)>1$, and let all the zeros of $A_k(z)$ lie in the angular sector  $\theta_1<\arg z<\theta_2$ satisfying
$$\theta_2-\theta_1<\frac{\pi}{p+1},$$ if $p$ is odd, and
$$\theta_2-\theta_1<\frac{(2p-1)\pi}{2p(p+1)},$$ if $p$ is even. Here, $p$ is the genus of $A_k(z)$. Suppose $A_0(z)$ is a transcendental entire function with a multiply-connected Fatou component satisfying $\rho(A_j)<\rho(A_0)$, where $j= 0,1,\ldots,k-1,k+1,\ldots, m$. Then, all non-trivial solutions of equation \eqref{cdehigherorder} are of infinite order.
\end{theorem}
\begin{proof}
Suppose that $f$ is a non-trivial solution of equation \eqref{cdehigherorder} and $\rho(f)<\infty$. Let us rotate the axes of the complex plane and assume that for some $\epsilon>0$, all the zeros of $A_k(z)$ have their arguments in the set $S(p,\epsilon)$. Thus from Lemma \ref{lkwon}, there exists a positive real number $R_2$ such that for all $r>R_2$ and $\alpha >1$,
\begin{equation}\label{eqkwonhigherorder}
\min_{|z|=r}|A_k(z)|\leq|A_k(-r)|\leq\exp(-\alpha r^p)< 1.
\end{equation}

From equation \eqref{cdehigherorder},
$$|A_0(z)|\leq\left|\frac{f^{(m)}(re^{\iota\theta})}{f(re^{\iota\theta})}\right|+|A_{m-1}(z)|\left|\frac{f^{(m-1)}(re^{\iota\theta})}{f(re^{\iota\theta})}\right|+\cdots+|A_{k}(z)|\left|\frac{f^{(k)}(re^{\iota\theta})}{f(re^{\iota\theta})}\right|+\cdots+|A_0(z)|\left|\frac{f'(re^{\iota\theta})}{f(re^{\iota\theta})}\right|.$$
Using equations \eqref{p11}, \eqref{orderA_j}, \eqref{fatoueqhigherorder} and \eqref{eqkwonhigherorder},
$$M(r,A_0)^{\beta}<|A_0(z)|\leq (1+(m-1)\exp{r^{\rho(A_t) + \epsilon}}+o(1))r^{m\rho(f)}$$
for large $r\in H\setminus (E_1\cup[0,1])$ and $\theta\in\{\theta:\min_{|z|=r}|A_k(z)|=|A_k(z)|\}$. But, we have $\rho(A_j)<\rho(A_0)$, $j=1,2,\cdots,k-1,k+1,\cdots, m$.
\end{proof}

\end{document}